\documentclass[12pt,reqno]{amsart}
\usepackage{amssymb}

\newcommand{\e}{\varepsilon}

\newcommand{\al}{\alpha}
\newcommand{\te}{\theta}
\newcommand{\fy}{\varphi}

\newcommand{\p}{\partial}
\newcommand{\I}{\infty}

\newcommand{\R}{\mathbb{R}}

\newcommand{\N}{\mathbb{N}}
\newcommand{\Z}{\mathbb{Z}}

\newcommand{\expo}{\operatorname{e}}

\numberwithin{equation}{section}

\newtheorem{thm}{Theorem}[section]

\newtheorem{lem}[thm]{Lemma}
\newtheorem{prop}[thm]{Proposition}

\theoremstyle{remark}
\newtheorem{rem}{Remark}

\newcommand{\lec}{\lesssim}
\newcommand{\gec}{\gtrsim}

\newcommand{\EQ}[1]{\begin{equation} \begin{split} #1 \end{split} \end{equation}}
\newcommand{\Del}[1]{}
\newcommand{\CAS}[1]{\begin{cases} #1 \end{cases}}
\newcommand{\pt}{&}
\newcommand{\pr}{\\ &}
\newcommand{\pq}{\quad}

\newcommand{\de}{\delta}

\newcommand{\ka}{\kappa}

\newcommand{\x}{\xi}
\newcommand{\y}{\eta}
\newcommand{\z}{\zeta}

\newcommand{\na}{\nabla}

\newcommand{\De}{\Delta}
\newcommand{\Om}{\Omega}

\newcommand{\Limsup}{\varlimsup}

\begin{document}
\newcommand{\frab}[2]{\left[\frac{#1}{#2}\right]}



\author{Slim Ibrahim}
\address{Department of Mathematics and Statistics \\ University of Victoria \\
 PO Box 3060 STN CSC \\ Victoria, BC, V8P 5C3\\ Canada}
\email{ibrahim@math.uvic.ca}
\thanks{S. I.  is partially supported by NSERC\#
 371637-2009 grant and a start up fund from University of 
Victoria.}


\author{Nader Masmoudi}
\address{New York University \\
The Courant Institute for Mathematical Sciences.}
\email{masmoudi@courant.nyu.edu}
\thanks{N. M is partially supported by an NSF Grant DMS-0703145}


\author{Kenji Nakanishi}
\address{Department of Mathematics, Kyoto University}
\email{n-kenji@math.kyoto-u.ac.jp}

\title[Trudinger-Moser with the exact growth]
{Trudinger-Moser inequality on the whole plane with the exact growth condition}


\keywords{Sobolev critical exponent, Trudinger-Moser inequality, 
concentration compactness, nonlinear Schr\"odinger equation, ground state.}

\subjclass[2010]{35J20,46E35,35Q55} 

\begin{abstract}
Trudinger-Moser inequality is a substitute to the (forbidden) critical Sobolev embedding, namely the case where the scaling corresponds to $L^\infty$. 
It is well known that the original form of the inequality with the sharp exponent (proved by Moser) fails on the whole plane, but a few modified versions are available. We prove a precised version of the latter, giving {\it necessary and sufficient conditions} for the boundedness, as well as for the compactness, in terms of the growth and decay of the nonlinear function. 
It is tightly related to the ground state of the nonlinear Schr\"odinger equation (or the nonlinear Klein-Gordon equation), for which the range of the time phase (or the mass constant) as well as the energy is given by the best constant of the inequality. 
\end{abstract}

\maketitle 


\section{Introduction}
 

There are several extentions of the critical Sobolev embedding
\EQ{ \label{Sobolev high D}
 \dot H^1(\R^d)\subset L^{2d/(d-2)}(\R^d)} 
from $d \geq 3$ to $d=2$, where the simple limit estimate fails
\EQ{
 H^1(\R^2)\not\subset L^\I(\R^2).} 
One way is to replace the right hand side by $BMO$, Besov, Triebel-Lizorkin or Morrey-Companato spaces of the same type (scaling). These are all taking account of possible oscillations of the functions in $H^1(\R^2)$. If one is interested more in possible growth, another substitute is given by the Trudinger-Moser inequality \cite{Trudinger67,Moser71,Po,Yodovich61} 
\begin{prop}\label{proptm3} For any open set $\Omega \subset \R^2$ with bounded measure $|\Om|<\I$, there exists a constant $\kappa(\Omega)>0$ such that 
\begin{equation} \label{Mos3} 
 u \in H^1_0(\Omega),\ \|\nabla u\|_{L^2(\Om)}\leq 1 \implies \int_{\R^2}\,\left(\expo^{4\pi |u|^2}-1\right)\,dx \le \ka(\Om). 
\end{equation} 
Moreover, this fails if $4\pi$ is replaced with any $\alpha>4\pi$. The constant $\kappa(\Omega)$ is bounded by $|\Omega|$, but in general unbounded as $|\Om|\to\I$. 
\end{prop}
The goal of this paper is to give a precised version of this inequality in the whole space $\R^2$, with necessary and sufficient conditions in terms of the growth of general nonlinear functionals (not only for $e^{\al|u|^2}-1$). 
Before stating our result, let us first recall the following two versions of the Trudinger-Moser inequality on $\R^2$. The first one is for smaller exponents. 
\begin{prop} \label{mostrud} For any $\alpha<4\pi$, there exists a constant $c_\alpha>0$ exists such that
\begin{equation}\label{Mos1} 
 u\in H^1(\R^2),\ \|\nabla u\|_{L^2(\R^2)}\leq1 \implies \int_{\R^2}\,\left(\expo^{\alpha |u|^2}-1\right)\,dx\leq c_\alpha \|u\|_{L^2(\R^2)}^2. 
\end{equation}
Moreover, this fails if $\alpha$ is replaced with $4\pi$.
\end{prop}
One can normalize to $\|u\|_{L^2}=1$ by scaling. This version was proved in \cite{Cao}, using the symmetric decreasing rearrangement as Moser did \cite{Moser71}. The necessity $\al<4\pi$ was proved in \cite{AT}, also using Moser's example. 

The second one is to strengthen the condition on $\|\na u\|_{L^2(\R^2)}$ to the whole $H^1(\R^2)$ norm. Then the value $\alpha=4\pi$ becomes admissibe. 
\begin{prop} \label{proptm}
There exists a constant $\ka>0$ such that 
\begin{equation}\label{Mos2} 
 u\in H^1(\R^2),\ \|u\|_{H^1 (\R^2)}\leq 1 \implies \int_{\R^2}\,\left(\expo^{4\pi |u|^2}-1\right)\,dx\le \kappa.
\end{equation}
Moreover, this fails if $4\pi$ is replaced with any $\alpha>4\pi$.
\end{prop}
This version was proved in \cite{Ruf}, again by Moser's argument, while the failure for $\al>4\pi$ is clear from the sharpness in the previous two propositions.

In short, the failure of the original Trudinger-Moser \eqref{Mos3} on $\R^2$ can be recovered either by weakening the exponent $\al=4\pi$ or by strengthening the norm $\|\na u\|_{L^2}$. It is worth noting, however, that proving these two estimates on $\R^2$ is considerably easier than the critical case $4\pi$ on $\Om$, which suggests that there is some room of improvement, even though they are ``sharp" in their formulations. 

Then a natural question arises, 
\EQ{
 \text{\it What if we keep both the conditions $\alpha = 4\pi$ {\it and} $\|\na u\|_{L^2(\R^2)}\leq 1$?} }  
Our answer to this question is to weaken the exponential nonlinearity as follows: 
\begin{prop} \label{TM4}
There exists a constant $c>0$ such that 
\begin{equation}\label{Mos4} 
 u\in H^1(\R^2),\ \|\na u\|_{L^2(\R^2)}\le 1 \implies 
 \int_{\R^2}\,\frac{\expo^{4 \pi|u|^2}-1}{(1+|u|)^2}\,dx\leq c \|u\|_{L^2(\R^2)}^2. 
\end{equation}
Moreover, this fails if the power $2$ in the denominator is replaced with any $p<2$. 
\end{prop}
Obviously this implies Proposition \ref{mostrud}. It is less obvious that it also implies Proposition \ref{proptm}, but indeed by H\"older only; see Section \ref{Mos4 to Mos2}. Hence Proposition \ref{TM4} can be regarded as a unified improvement of those two previous versions, while it can be easily deduced from the part (B) of our full Theorem \ref{main} below, by taking 
\EQ{
 2\pi K = 1, \pq g(u) =\frac{ (\expo^{4 \pi|u|^2}-1)}{(1+|u|)^2}.}  
The following theorem completely determines the growth order, not only among exponentials and power functionals, but for general functions, in terms of necessary and sufficient conditions, both for the boundedness and for the compactness. 

\begin{thm} \label{main}
For any Borel function $g:\R\to[0,\I)$, define the functional $G$ by 
\EQ{
 G(\fy)=\int_{\R^2}g(\fy(x))dx.} 
Then for any $K>0$ we have the following {\bf (B)} and {\bf (C)}. 

\noindent {\bf (B)} Boundedness: The following (1) and (2) are equivalent. 
\begin{enumerate}
\item $\Limsup_{|u|\to\I} \expo^{-2|u|^2/K}|u|^2g(u)<\I$ and $\Limsup_{u\to 0}|u|^{-2}g(u)<\I$.
\item There exists a constant $C_{g,K}>0$ such that 
\EQ{ \label{TMp}
 \fy\in H^1(\R^2),\ \|\na\fy\|_{L^2(\R^2)}^2\le 2\pi K \implies G(\fy)\le C_{g,K}\|\fy\|_{L^2}^2.}
\end{enumerate}
Moreover, if (1) fails then there exists a sequene $\fy_n\in H^1(\R^2)$ satisfying 
\EQ{ 
 \|\na\fy_n\|_{L^2}^2<2\pi K, \pq \|\fy_n\|_{L^2}\to 0, \pq G(\fy_n)\to\I \pq(n\to\I).}

\noindent{\bf (C)} Compactness: The following (3) and (4) are equivalent. 
\begin{enumerate} \setcounter{enumi}{2}
\item $\lim_{|u|\to\I} \expo^{-2|u|^2/K}|u|^2g(u)=0$ and $\lim_{u\to 0}|u|^{-2}g(u)=0$. 
\item For any sequence of radial $\fy_n\in H^1(\R^2)$ satisfying $\|\na\fy_n\|_{L^2}^2\le 2\pi K$ and weakly converging to some $\fy\in H^1(\R^2)$, we have $G(\fy_n)\to G(\fy)$. 
\end{enumerate}
Moreover, if (3) fails then there exists a sequence of radial $\fy_n\in H^1(\R^2)$ satisfying $\|\na\fy_n\|_{L^2}^2<2\pi K$, converging to $0$ weakly in $H^1$, and $G(\fy_n)>\de$ for some $\de>0$. 
\end{thm}

\begin{rem}
1) This theorem shows that 
 the true threshold for the $L^2$ Trudinger-Moser inequality in the whole space under the condition $\|\nabla \fy \|_{L^2(\R^2)}\leq1 $ is given by the functional in \eqref{Mos4}, and concentration of compactness happens only for it. 

2)  As we will see in the proof, 
 the concentration sequence constructed in (3)  is very different from those in the higher dimensional case---it must contain nontrivial tail going to the spatial infinity which is the main contribution to $L^2$, even though the main contribution to $G$ is
 the concentrating part. This is also different from the concentration in the original Trudinger-Moser inequality \eqref{Mos3} (see \cite{Lions1}). 

3) In the Orlicz space corresponding to $g(u)\sim \expo^{\al|u|^2}$, \cite{BMM11} gives a more precise description of concentration compactness, in terms of the profile decomposition. However, the above phonomena for $g(u)\sim \expo^{2|u|^2/K}|u|^{-2}$ do not seem to be observable in a linear setting such as in the Orlicz space. 
\end{rem}
 
In Section \ref{conc}, we prove the necessity of (1) and (3) in Theorem \ref{main}, by constructing sequences $\fy_n$ obtained from rescaling of Moser's example. 
In Section \ref{prel}, we study the optimal growth of a function in the exterior of the ball when the $L^2$ and $\dot H^1$ norm are given. 
This is used in proving the sufficiency or the main part of Theorem \ref{main}  in Section \ref{24}. 

\section{Proof of the necessity of (1) and (3): Moser's example} \label{conc}
First we consider the much easier case with the condition as $|u|\to 0$. Let $\fy_n(x)$ be a sequence of radial functions in $H^1(\R^2)$ defined by 
\EQ{
 \fy_n(x)=\CAS{ a_n &(|x|<R_n),\\ a_n(1-|x|+R_n) &(R_n<|x|<R_{n}+1)\\ 0 &(|x|>R_n+1)}}
for some sequences $a_n\to 0$ and $R_n\to\I$ chosen later. We have 
\EQ{
 \pt \|\fy_n\|_{L^2}^2 \sim a_n^2R_n^2, 
 \pq \|\na\fy_n\|_{L^2}^2 \sim a_n^2R_n, 
 \pq G(\fy_n) \ge \pi R_n^2 g(a_n).}

If (1) is violated by $\Limsup_{|u|\to 0}|u|^{-2}g(u)=\I$, then we can find a sequence $a_n\searrow 0$ such that $g(a_n)\ge n|a_n|^2$. Let $R_n=a_n^{-1/2}+a_n^{-1}n^{-1/4}$. 
Then $R_n\to\I$, $a_nR_n\to 0$ and $G(\fy_n)\ge na_n^2R_n^2\to\I$. 

If (3) is violated by $\Limsup_{|u|\to 0}|u|^{-2}g(u)>0$, then we can find a sequence $a_n\searrow 0$ and $\de>0$ such that $g(a_n)\ge \de|a_n|^2$. Let $R_n=1/a_n$. 
Then $R_n\to\I$, $a_nR_n=1$, $a_n^2R_n\to 0$ and $G(\fy_n)\ge \de a_n^2R_n^2 \ge \de$. 

It remains to treat the case where the condition for $|u|\to\I$ fails. 
First, we recall the following fundamental example of Moser:  Let  $f_{\alpha}$ be  defined by:
\begin{eqnarray*}
 f_{\alpha}(x)&=&\; \left\{
\begin{array}{cllll}0 \quad&\mbox{if}&\quad
|x|\geq 1,\\\\ -\frac{\log|x|}{\sqrt{2\alpha\pi}} \quad
&\mbox{if}&\quad \expo^{-\alpha}\leq |x|\leq 1 ,\\\\
\sqrt{\frac{\alpha}{2\pi}}\quad&\mbox{if}&\quad |x|\leq {\rm
e}^{-\alpha},
\end{array}
\right.
\end{eqnarray*} 
where $\alpha>0$.  One can also write $f_\alpha$ as 
\EQ{
f_{\alpha}(x)=\sqrt{\frac{\alpha}{2\pi}}\,{\mathbf
L}\left(\frac{- \log|x|}{\alpha}\right),
}
where
\begin{eqnarray*}
{\mathbf L}(t)&=&\; \left\{
\begin{array}{cllll}0 \quad&\mbox{if}&\quad
t\leq 0,\\ t \quad
&\mbox{if}&\quad 0\leq t\leq 1 ,\\
1 \quad&\mbox{if}&\quad t\geq 1.
\end{array}
\right.
\end{eqnarray*}
  Straightforward  computations show that
$\|f_{\alpha}\|_{L^2(\R^2)}^2=\frac{1}{4\alpha}(1-{\rm
e}^{-2\alpha})-\frac{1}{2}\expo^{-2\alpha}$ and $\|\nabla
f_{\alpha}\|_{L^2(\R^2)}=1$. 

In order to fit this example in our estimate on $\R^2$, we need to make sure that the $L^2$ norm does not go to zero. This requires a rescaling of Moser's example. 
Choose sequences $1\ll b_k\nearrow\I$ and $K_k\nearrow K$ such that 
\EQ{
 c_k:=\expo^{-2 b_k^2/K_k}b_k^2 g(b_k)\to\Limsup_{|u|\to\I}\expo^{-2|u|^2/K}|u|^2g(u),}
and let $R_k=e^{-  b_k^2/K_k}$.  We define  a radial function  $\psi_k(r)\in H^1 (\R^2) $ by 
\EQ{
 \psi_k(r) = \CAS{b_k &(r<R_k) \\ b_k\frac{|\log r|}{|\log R_k|} & (R_k\le r<1) \\ 0 &(r\ge 1)}}
Then we have
\EQ{
 \pt \|\na\psi_k\|_{L^2}^2 = 2\pi b_k^2 \int_{R_k}^1 \frac{dr}{r|\log R_k|^2} = 2\pi K_k< 2\pi K,
 \pr \|\psi_k\|_{L^2}^2 \sim \frac{2\pi b_k^2}{|\log R_k|^2} = \frac{2  \pi K_k^2}{  b_k^2}, 
 \pq G(\psi_k) \ge 2\pi R_k^2 g(b_k) = \frac{2\pi c_k}{b_k^2}, }
and so $G(\psi_k)\|\psi_k\|_{L^2}^{-2}\ge c_k/K^2$. 

Now let $\fy_k=\psi_k(x/S_k)$, where we choose $S_k=b_k$  if  (3) fails 
 (i.e.~if $c_k>0$), and $S_k=o(b_k)$ such that $S_k^2 c_k /b_k^2\to \I$  if  (1) fails  (i.e.~if $c_k=\I$). 
In both cases $\fy_k$ is bounded in $H^1$  and satisfies 
\EQ{ 
 \pt \|\na\fy_k\|_{L^2}^2 = \|\na\psi_k\|_{L^2}^2   = 2\pi K_k< 2 \pi K,
 \pr \|\fy_k\|_{L^2}^2  = S_k^2  \|\psi_k\|_{L^2}^2    \sim 2  \pi K_k^2   \frac{  S_k^2  }{  b_k^2}, 
 \pq G(\fy_k) \ge 2\pi R_k^2 g(b_k) = \frac{2\pi c_k   S_k^2  }{b_k^2} }
Moreover, 
 $\fy_k(x)\to 0$ for every $x\not=0$, because $|\fy_k(x)|\lec \e$ if $|x|\ge S_k\expo^{-\e b_k}=o(1)$ for any $\e>0$. 
This ends the proof for the necessity of (1) and (3). 

\section{Radial Trudinger-Moser and optimal descending} \label{prel}
In this section, we prove the following theorem that will be used in the proof of point (2) of the main  theorem \ref{main}. This can be regarded as the exponential version of the radial Sobolev inequality. 
\begin{thm} \label{rad TM}
There exists a constant $C>0$ such that for any radial $\fy\in H^1$, 
\EQ{ 
 \|\na\fy\|_{L^2(|x|>R)}^2\le 2\pi K \implies \frac{\expo^{2\fy(R)^2/K}}{\fy(R)^2/K^2} \le C\|\fy/R\|_{L^2(|x|>R)}^2.}
\end{thm}
The function $\expo^{\fy^2}/\fy$ is optimal in the above estimate. More precisely, we have 
\begin{thm} \label{opt desc}
Let 
\EQ{
 \pt \mu(h) := \inf\{\|\fy\|_{L^2(|x|>1)} \mid \fy\in H^1_{r},\ \fy(1)=h,\ \|\fy_r\|_{L^2(|x|>1)}^2 \le 2\pi\},}
for $h> 1$. Then we have $\mu(h)\sim \expo^{h^2}/h$ for $h> 1$.
\end{thm} 
Obviously, the first theorem follows from the second one, by rescaling. 
To prove the latter, we consider the descrete version: 
\EQ{
  \mu_d(h):=\inf\{ \|a\|_{(e)} \mid \|a\|_1=h,\ \|a\|_2\le 1\}.}
for $h> 1$, where the norms on any sequence $a=(a_n)_{n=0}^\I$ are defined by
\EQ{
 \|a\|_p^p = \sum_{n=0}^\I |a_n|^p, \pq \|a\|_{(e)}^2 = \sum_{n=0}^\I \expo^{2n} a_n^2.}
\begin{lem}
We have $\mu(h)\sim\mu_d(h)$ for $h> 1$.
\end{lem}
\begin{proof}
$\mu_d$ is naturally obtained by optimizing the energy for given values on the lattice. 
For $\mu(h)$, it suffices to consider radial $\fy\in H^1$ satisfying $\fy_r\le 0 \le \fy$. 

Let $h_k=\fy(\expo^k)$ and $a_k=h_k-h_{k+1}$ for $k=0,1,2\dots$. 
We can optimize the energy on each interval $[\expo^k,\expo^{k+1}]$ by replacing $\fy$ with 
\EQ{
 \psi(r) = a_k |\log(\expo^{-k-1}r)| + h_{k+1} \pq(\expo^k\le r\le \expo^{k+1}).}
Then we have $\psi(1)=h_0=\fy(1)=h$ and 
\EQ{
 \int_{\expo^k}^{\expo^{k+1}}\psi_r(r)^2 rdr = a_k^2 = (\fy(\expo^k)-\fy(\expo^{k+1}))^2
 \le \int_{\expo^k}^{\expo^{k+1}}\fy_r(r)^2 rdr,}
where the last inequality follows from Schwarz. For the $L^2$ norm we have 
\EQ{
 \|\psi\|_{L^2(r>1)}^2 \lec \sum_{k=0}^\I h_k^2 \expo^{2k} \lec h_0^2 + \int_0^\I \|\fy\|_{L^2(r>1)}^2,}
where $h_0^2$ is estimated by using the energy as follows. For $1<r<\expo^{1/4}$ we have 
\EQ{
 \fy(1)-\fy(r) \le \int_r^1\fy_r(s)ds \le \sqrt{\int_1^{\expo^{1/4}}|\fy_r|^2rdr \int_1^{\expo^{1/4}}\frac{dr}{r}} \le \frac{1}{2},}
which implies $\fy(r)\ge h_0/2$ since $h_0=h> 1$, and so 
\EQ{
 h_0^2 \lec \int_1^{\expo^{1/4}}|\fy|^2rdr \lec \|\fy\|_{L^2}^2.}

Hence for $\mu(h)$ it suffices to consider such $\psi$. Moreover we have 
\EQ{
 \|a\|_2^2 \le \|\psi_r\|_{L^2(|x|>1)}^2/(2\pi) \le 1, \pq 
 \|a\|_1 = h_0 = h,} 
and 
\EQ{
 \|\psi\|_{L^2(r>1)}^2 \sim \sum_{j\ge 0}h_j^2 \expo^{2j}
 \pt= \sum_{j} \sum_{k,l\ge j} a_k a_l \expo^{2j}
 \pr=\sum_{k,l} a_k a_l \sum_{j\le\min(k,l)}\expo^{2j}
 \sim \sum_{k\le l} a_k a_l \expo^{2k} \sim \|a\|_{(e)}^2,}
where $\lec$ for the last equivalence follows from Young on $\Z$. 
\end{proof}
Now,  Theorem \ref{opt desc}, and hence Theorem \ref{rad TM}, follow from 
\begin{lem}
For $h>1$ we have 
\EQ{
 \mu_d(h) \sim \frac{\expo^{h^2}}{h}.}
\end{lem}
This is essentially achieved by constant sequences of finite length, corresponding to the Moser's function in the continuous version. 
The $(e)$ norm determines the fall-off or the length of the sequence, and then optimization of the embedding $\ell^2\subset\ell^1$ on the finite length forces it to be a constant. 
\begin{proof}
Since $\mu_d(h)$ is increasing in $h$, it suffices to show $\mu_d(\sqrt{n})\sim \expo^{n}/\sqrt{n}$ 
for all integer $n$. 
$\lec$ is easily seen by choosing $a=(1,\dots,1)/\sqrt{n}$, so we consider $\gec$. Suppose by contradiction that for some $\e\ll 1$, $n\gg 1$ and sequence $a$ we have 
\EQ{
 \|a\|_2\le 1, \pq \|a\|_1 = \sqrt{n}, \pq \|a\|_{(e)}^2\le \frac{\e^2  \expo^{2n}}{n}.}
From the last condition we get 
\EQ{
 n\le j \implies |a_j|\lec \frac{\e}{\sqrt{n}}\expo^{n-j},}
and so letting $a_j'=a_j$ for $j\le n$ and $a_j'=0$ for $j>n$, we get 
\EQ{ \label{l1 low bd}
 \|a'\|_1 \ge \|a\|_1-\sum_{j>n}|a_j| \ge \sqrt{n}-\frac{C\e}{\sqrt{n}}.}
Then by the support of $a'$ we have 
\EQ{
 n-C\e \le \|a'\|_1^2 = n\|a'\|_2^2 - \sum_{j,k\le n}(a_j-a_k)^2/2
 \le n - \sum_{j,k\le n}(a_j-a_k)^2/2,}
hence 
\EQ{
 \sum_{j,k\le n}(a_j-a_k)^2 \lec \e.}
Choose $m\le n$ so that $\min_{j\le n}|a_j|=|a_m|$. Then we get from the above estimate 
\EQ{
 \|a'\|_1 - n|a_m| \le \|a_j-a_m\|_{\ell^1(j\le n)}
 \le \sqrt{n}\|a_j-a_m\|_{\ell^2(j\le n)} \lec \sqrt{n\e}.}
Combining it with \eqref{l1 low bd}, we get
\EQ{
 |a_m| \gec \sqrt{n}/n = 1/\sqrt{n},}
provided that $\e>0$ is small enough. Since $|a_n|\ge|a_m|$, we obtain $\|a\|_{(e)}\gec \expo^n/\sqrt{n}$,  
which yields a contradiction. 
Hence, we deduce that  $\mu_d(h)^2 \sim \expo^{2h^2}/h^2$. 
\end{proof}

\section{Proof of (2) and (4) of Theorem \ref{main}} \label{24}
\begin{proof}

To prove (2) of the theorem \ref{main}, it suffices to show  that 
\EQ{\label{inG}
 G(\fy) = \int_{\R^2} \min(|\fy|^2,|\fy|^{-2})\expo^{2|\fy|^2} dx \lec \|\fy\|_{L^2}^2,}
for all non-negative, radially decreasing $\fy\in H^1(\R^2)$ satisfying $\|\na\fy\|_{L^2}^2=2\pi$. 
Here we took $K=1$. 
 Fix such a radial function $\fy(x)=\fy(r)$ and let $g(s) = \min(|s|^2,|s|^{-2})\expo^{2|s|^2}  $. 

Choose $R_0>0$ such that $\|\na\fy\|_{L^2(r>R_0)}^2=2\pi K_0$, where $K_0 = \ka\in(2/3,1)$ is a constant which will be determined later. 
The region $r>R_0$ is subcritical and easily estimated. 
Indeed, if $\fy(r) \leq 2$ for all $r \geq R_0$, then  \eqref{inG} in the region 
$\{  r > R_0 \}$ follows from  
 $g(\fy)\lec|\fy|^2$ for $|\fy| \lec 2 $. If not,  
we  choose $R>R_0$ so that $\fy(R)=1$.  For $R_0<r<R$ we have by Schwarz 
\EQ{
 \fy(r)-1=\int_r^R|\fy_r| dr\le \sqrt{K_0\log(R/r)}.}
Let $K'=(1+K_0)/2$. Then 
\EQ{
  \frac{\fy(r)^2}{K'} \le \frac{(\fy(r)-1)^2}{K_0}+ \frac{1}{K'-K_0} \le \log(R/r)+\frac{2}{1-K_0},}
and so
\EQ{
 \pt\int_{R_0}^R \expo^{2\fy(r)^2} rdr \lec \int_{R_0}^R (R/r)^{2K'}rdr 
 \lec R^2-(R_0)^2 \lec \int_{R_0}^R |\fy(r)|^2 rdr}
where we have used that $\fy(r) \geq 1$ in the region $R_0 \leq r \leq R$. 
Now, since $g(\fy(r))\lec|\fy(r)|^2$ for $r>R$, we thus obtain  that 
\EQ{
 \pt\int_{R}^\infty g(\fy(r))  rdr \lec \int_{R}^\infty |\fy(r)|^2 rdr}
and hence 
\EQ{
 \int_{|x|>R_0} g(\fy(x))dx \lec \int_{|x|>R_0}|\fy(x)|^2 dx = 2\pi M_0.}

Now we proceed to the main part $r<R_0$. Let 
\EQ{
 \pt R_j=R_0\expo^{-j}, \pq h_j=\fy(R_j), \pq a_j=\sqrt{\int_{R_j}^{R_{j-1}}|\fy_r|^2 rdr},
 \pr K_j=\int_{R_j}^\I|\fy_r|^2 rdr, \pq M_0=\int_{R_0}^\I |\fy|^2rdr.}
Then Theorem \ref{rad TM} gives 
\EQ{
 \frac{\expo^{2h_0^2/K_0}}{h_0^2}R_0^2 \lec M_0/K_0^2 \sim M_0.}
By the monotone convergence theorem, we may assume that $\fy$ is constant on $|x|<R_N$ for some $N\in\N$. Then it suffices to show that 
\EQ{ \label{red sum}
 \sum_{j=0}^N \frac{\expo^{2h_j^2}}{h_j^2}R_j^2 \lec M_0.}
First we derive a bound for each $j$. By Schwarz inequality,  we have 
\EQ{
 h_j-h_{j-1} = \int_{R_j}^{R_{j-1}}|\fy_r| dr \le a_j,}
and so, using that $K_j=K_{j-1}+a_j^2$, we get that 
\EQ{ \label{young}
 h_j^2 \le (h_{j-1}+a_j)^2 = \frac{K_jh_{j-1}^2}{K_{j-1}} + K_j -\frac{(a_jh_{j-1}-K_{j-1})^2}{K_{j-1}} \le \frac{K_jh_{j-1}^2}{K_{j-1}} + K_j.}
Hence,  
$$   \frac{h_{j}^2}{K_{j}}   \leq     \frac{h_{j-1}^2}{K_{j-1}} + 1.  $$
Now,  we define $H_j$ and $\x_j$ by 
\EQ{
 H_0=h_0, \pq H_j=H_{j-1}+a_j, \pq \x_j=H_j^2/K_j.}
Then we have 
\EQ{
 h_j\le H_j, \pq \x_j \le \x_{j-1}+1 \le \x_0 + j,}
which implies that 
\EQ{
 \y_j:=\expo^{2H_j^2/K_j}R_j^2=\expo^{2(\x_j-j)}R_0^2}
is monotone decreasing in $j$. This is not sufficient to sum over $j$, for which we have to sharpen the above estimate. 

The idea is to exploit the room given by the factor $1/K_j$ in the exponential to show that the sum \eqref{red sum} is essentially dominated by the first term $\y_0$. The  possible growth 
of the  denominator $h_j^2$ will not play any role for the summability and can be replaced by $h_0^2$. 

Let $J=\{1,\dots,N\}$ and  define the sets $A$ and $B$ by 
\EQ{
 A:=\{j\in J\mid a_jH_{j-1} \le K_{j-1}-K_j/3\}, \pq B:=J\setminus A.}
On the region $A$, the sequence $\y_j$ decays fast enough to be summed without using the factor $1/K_j$, while on $B$, the decrease of $1/K_j$ is effective enough to supply the summability.

Indeed, for $j\in A$, we have by the same computation as in \eqref{young}, 
\EQ{
 \x_j \le \x_{j-1}+8/9,}
whereas for $j\in B$ we have 
\EQ{
 a_j^2 \gec \frac{K_{j-1}^2}{H_{j-1}^2} = \frac{K_{j-1}}{\x_{j-1}} \gec \frac{1}{\x_0+j-1}.}
For the sum over $A$ we have 
\EQ{
 \sum_{j\in A} \y_j \le \sum_{k=1}^{\# A}\expo^{-2k/9}\y_0 \lec \y_0=\expo^{2h_0^2/K_0}R_0^2,}
and so
\EQ{ \label{bound over A}
 \sum_{j\in A} \frac{\expo^{2H_j^2/K_j}}{h_j^2}R_j^2 \lec \frac{\expo^{2h_0^2/K_0}}{h_0^2}R_0^2 \lec M_0.}
To bound the sum over $B$, let $a,a+1,\dots,b$ be any maximal consecutive sequence in $B$. Then for any $j\in\{a,\dots,b\}$ we have 
\EQ{
 \pt \x_j \le \x_a + j-a, \pq a_{j+1}^2 \ge \frac{\de}{\x_a+j-a},}
for some fixed  constant $\de>0$. Now let 
\EQ{
 \z_j := K_j(\x_a+j-a)-j.}
Then we have $H_j^2-j\le \z_j$ and 
\EQ{
 \z_j-\z_{j-1} = a_j^2(\x_a+j-1-a) + K_j-1 \ge \de + \ka -1.}
Now we choose $\ka$ sufficiently close to $1$ so that the right hand side is bigger than $\de/2$. Then we have 
\EQ{
 \sum_{j=a}^b \expo^{2H_j^2}R_j^2 \le \sum_{j=a}^b \expo^{2\z_j}R_0^2 \le \sum_{k=0}^{b-a} \expo^{2\z_b-\de k}R_0^2 \lec \expo^{2\z_b}R_0^2 \le \y_a \le \y_{a-1},}
and so
\EQ{
 \sum_{j\in B}\expo^{2H_j^2}R_j^2 \lec \y_0+\sum_{j\in A}\y_j\lec\y_0,}
which implies, together with \eqref{bound over A}, the desired estimate in (2).  

Finally we prove (4) from (2). By the radial Sobolev inequality we have, 
\EQ{
 |\fy(r)|^2 \lec \|\fy\|_{L^2}\|\fy_r\|_{L^2}/r,}
 and hence, $\fy_n(r)\to 0$ as $r\to\I$ uniformly in $n$. Moreover,   
\EQ{
 [\fy_n]_{R_0}^{R_1} = \int_{R_0}^{R_1}\p_r\fy_n dr}
converges as $n\to\I$ for any $0<R_0<R_1<\I$. Hence $\fy_n(r)\to\fy(r)$ at every $r\in(0,\I)$. 
By the radial Sobolev inequality and that $g(\fy)=o(|\fy|^2)$ as $\fy\to 0$, for any $\e>0$ there is $R>0$ independent of $n$ such that 
\EQ{ 
 \int_R^\I g(\fy_n) rdr \le \int_R^\I \e|\fy_n|^2rdr \lec \e\|\fy_n\|_{L^2}^2.}
 By (2) and the fact  that $g(\fy)=o(\exp(2|\fy|^2)|\fy|^{-2})$ as $|\fy|\to\I$, there is $L>1$ indepedent of $n$ such that 
\EQ{
 \int_{|\fy_n|>L}g(\fy_n) dx \le \int_{|\fy_n|>L}\e\exp(2|\fy_n|^2)|\fy_n|^{-2}dx \lec \e\|\fy_n\|_{L^2}^2.}
Now define $g^L$ by 
\EQ{
 g^L(u)=\CAS{g(u) &(|u|\le L) \\ g(Lu/|u|) &(|u|>L)}.}
Then we have 
\EQ{
 \Limsup_{n\to\I} |G(\fy_n)-G(\fy)| \lec \e+\Limsup_{n\to\I}\int_{|x|<R}|g^L(\fy_n)-g^L(\fy)|dx = \e,}
by the dominated convergence theorem. This ends the proof of (4). 
\end{proof}

\section{Elliptic equation}
Now we consider the existence of solutions for the nonlinear elliptic equation
\EQ{ \label{stat eq}
 -\De Q + c Q = f'(Q),}
for $Q(x):\R^2\to\R$, $c>0$, and exponential nonlinearity $f$. The existence of the ground state, namely the positive radial solution, with the least energy among all solutions, has been studied by the ODE technique in \cite{AP,GMSTY} including supercritical nonlinearity, and by the variational technique in \cite{BGK} for subcritical nonlinearity, and in \cite{Cao} including the critical nonlinearity. 

Combining the above precised Trudinger-Moser inequality with the argument in \cite{ScatBlow}, we deduce the following. Let $D$ be the operator defined by $Df(u)=uf'(u)$. 

\begin{thm}
Assume that $f:\R\to\R$ satisfies $f(u)=o(u^2)$ as $u\to 0$, $(D-2)f\ge\e f\ge 0$ for some $\e>0$, and that there exists $\ka_0\ge 0$ such that for all $\ka_+>\ka_0>\ka_-$, 
\EQ{
 \lim_{|u|\to\I}Df(u)\expo^{-\ka_+|u|^2}=0, \pq 
 \lim_{|u|\to\I}f(u)\expo^{-\ka_-|u|^2}=\I,}
and $\lim_{|u|\to\I}Df(u)/f(u)=\I$. 
Then there exists $c_*\in(0,\I]$ such that there is a positive radial solution $Q_c$ for each $c\in(0,c_*]$, which has the least energy among all the solutions. Moreover, $c_*=C_{f,\ka_0}$ in (2) of Theorem \ref{main} when it is finite, while $c_*=\I$ is equivalent to 
\EQ{
 \Limsup_{|u|\to\I} f(u)\expo^{-\ka_0|u|^2}|u|^2 = \I.}
In addition, we have 
\EQ{
 \ka_0\|\na Q_c\|_{L^2}^2 \le 4\pi,}
where the equality holds if and only if $c=c_*$. 
\end{thm}
The subcritical nonlinearity is covered by taking $\ka_0=0$, while the supercritical nonlinearity such as those in \cite[Theorem 6]{GMSTY} is also covered by taking $\ka_0>0$. Although it is not easy to fully compare the conditions in \cite{AP,GMSTY} with our variational condition, there is certainly a new case, that is when 
\EQ{
 f(u)\gg e^{\ka_0|u|^2}/|u|^2, \pq f'(u)=o(e^{\ka_0|u|^2}) \pq(|u|\to\I).} 
Such nonlinearity has been investigated for the Dirichlet problem 
\EQ{
 -\De u = f'(u) \text{ in $\Om$,} \pq u=0 \text{ on $\partial\Om$}}
on a bounded domain $\Om$, including the threshold case $f'(u)=e^{\ka_0|u|^2}/u$, see \cite{FOR}. The ground state for $c=1$ was constructed in \cite{Cao} under the conditions $|f'(u)|\lec e^{4\pi u^2}$, $(D-2)f\ge\e f$ and, for some $p\in(2,\I)$, 
\EQ{
 Df(u) \ge \frac{p}{2}|u|^p\left(\frac{\e}{2+\e}\right)^{1-p/2}\inf_{0\not=u\in H^1(\R^2)}\frac{\|u\|_{H^1}^p}{\|u\|_{L^p(\R^2)}^p}.}
It seems difficult to compare this and our condition $1\le C_{f,\ka_0}$. 

We omit a proof of the above theorem, for it was completely proved in \cite{ScatBlow} except for the necessary and sufficient condition for $c_*=\I$, but it was stated in the form (2) on Theorem \ref{main}. 
It is worth noting that the compactness (4) of Theorem \ref{main} does not seem useful for the above problem, but the compactness on a minimizing sequence comes from the superpower growth $Df(u)/f(u)\to\I$ together with a variational constraint, see \cite{ScatBlow}. A related fact is that the best constants in \eqref{Mos3} and in \eqref{Mos2} are not attained by any concentrating sequences, see \cite{CC,Flu,Ruf}. 

\section{From TM with the exact growth to TM with $H^1$}\label{Mos4 to Mos2}
In this section we show how one can derive Proposition \ref{proptm} from our inequality \eqref{Mos4} using H\"older only. Let $u\in H^1(\R^2)$ satisfy $\|\na u\|_{L^2}\le 1$. First observation is that by Taylor expansion of $\exp$, there is a constant $C_0>0$ such that 
\EQ{ \label{power bound}
 \forall n\in\N, \pq \int_{\R^2}(4\pi|u|^2)^n dx\le C_0(n+1)!\|u\|_{L^2}^2,}
hence there is a constant $C_1>0$ such that for any $p\ge 1$
\EQ{ \label{L2p bd}
 \||u|^2\|_{L^p} \le C_1 p \|u\|_{L^2}^{2/p},}
which is extended to noninteger $p$ by H\"older. 

Next, if $\|u\|_{H^1}^2\le 1$ then for some $\te\in(0,1)$ we have
\EQ{
 \|u\|_{L^2}^2 \le \te, \pq \|\na u\|_{L^2}^2 \le 1-\te.}
If we can take $\te\ge 1/2$, then \eqref{Mos2} follows from \eqref{Mos1} applied to
  $\sqrt{2} u $ with $\al=2\pi$. So we may assume that $\te<1/2$. 
Let $A:=\{x\in\R^2\mid |u(x)|\ge 1\}$. Then by H\"older 
\EQ{ \label{Mos2Hol}
 \int_A e^{4\pi|u|^2}dx \le \left[\int_A \frac{e^{4\pi|u|^2/(1-\te)}}{|u|^2/(1-\te)}dx\right]^{1-\te}\||u|^2/(1-\te)\|_{L^{\frac{1-\te}{\te}}(A)}^{1-\te},}
where the first term on the right is bounded by \eqref{Mos4}
\EQ{
 \int_A \frac{e^{4\pi|u|^2/(1-\te)}}{|u|^2/(1-\te)}dx \le C_2\frac{\|u\|_{L^2}^2}{1-\te} \le C_2\frac{\te}{1-\te},}
while the second term is bounded by \eqref{power bound}
\EQ{
 \left\|\frac{|u|^2}{1-\te}\right\|_{L^{\frac{1-\te}{\te}}(A)} \le C_1\frac{1-\te}{\te}\left\|\frac{u}{\sqrt{1-\te}}\right\|_{L^2}^{\frac{2\te}{1-\te}} \le C_1\left[\frac{1-\te}{\te}\right]^{\frac{1-2\te}{1-\te}}. }
Putting them into \eqref{Mos2Hol}, we obtain 
\EQ{ 
 \int_A e^{4\pi|u|^2}dx \lec  \left[ \frac{\te}{1-\te} \right]^{1-\te} 
  \left[\frac{1-\te}{\te}\right]^{1-2\te}  \le \left[ \frac{\te}{1-\te} \right]^\te,}
which is bounded for $0<\te<1/2$. The estimate on $\R^2\setminus A$ is trivial. 

\section{Acknowledgments}
The first author is thankful to Professor Yoshio Tsutsumi and all members of the Math Department at Kyoto University for their very generous hospitality.


\begin{thebibliography}{10}
\bibitem{AT}
S.~Adachi and K.~Tanaka, {\em Trudinger type inequalities in
$\mathbb R^N$ and their best exponents}, Proc.~Amer.~Math.~Soc. {\bf
128} (2000), no.~7, 2051--2057.

\bibitem{AP}
F.~V.~Atkinson and L.~A.~Peletier, {\em Ground states and Dirichlet problems for $-\De u = f (u)$ in $R^2$}, Arch. Rational Mech. Anal. {\bf 96} (1986), 147--165. 

\bibitem{BG}
H.~Bahouri and P.~G\'erard, {\em High frequency approximation of
solutions to critical nonlinear wave equations}, Amer. J. Math. {\bf
121} (1999), 131--175.


\bibitem{BMM11}
H.~Bahouri, M.~Majdoub, and N.~Masmoudi.
\newblock On the lack of compactness in the 2{D} critical {S}obolev embedding.
\newblock {\em J. Funct. Anal.}, 260(1):208--252, 2011.


\bibitem{BGK}
H.~Berestycki, T.~Gallouet, and O.~Kavian, {\em \'Equations de champs scalaires euclidiens non-lin\'eaires dans le plan}, C. R. Acad. Sci. Paris {\bf 297} (1983), 307--310. 

\bibitem{BL}
H.~Berestycki and P.L.~Lions, {\em Nonlinear Scalar field equations,
I. Existence of ground state}, Arch. Rat. Mech. Anmal., {\bf 82}
(1983), 313--346.

\bibitem{BW} H. Br{\'e}zis and S. Wainger,  {\em A note on limiting cases of {S}obolev embeddings and convolution inequalities},  {Comm. Partial Differential Equations}, {\bf 5} (1980), 773--789.

\bibitem{Cao} D.~M.~Cao, 
{\em Nontrivial solution of semilinear elliptic equation with critical exponent in $R^2$}, Comm. Partial Differential Equations {\bf 17} (1992), no.~3-4, 407--435.

\bibitem{CC}
L.~Carleson and A.~Chang, {\em On the existence of an extremal
function for an inequality of J. Moser}, Bull. Sc. Math. {\bf 110}
(1986), 113--127.

\bibitem{Cianchi} A. Cianchi, {\em A sharp embedding theorem for {O}rlicz-{S}obolev spaces}, {Indiana Univ. Math. J.}, {\bf 45}(1996), 39--65.

\bibitem{DR95}
D.~G.~De Figueiredo and B.~Ruf, 
{\it Existence and Non-Existence of Radial Solutions for
Elliptic Equations with Critical Exponent in $\R^2$}. 
Comm. Pure and Appl. {\bf 48} (1995), 639--655.

\bibitem{FOR}
D.~G.~De Figueiredo, J.~M.~do \'O and B.~Ruf, {\em On an inequality by N.~Trudinger and J.~Moser and related elliptic equation}, Comm. Pure Appl. Math. {\bf 55} (2002) 1--18. 

\bibitem{EKP} D. E. Edmunds, R. Kerman and L. Pick, {\em Optimal {S}obolev imbeddings involving rearrangement-invariant quasinorms},  {J. Funct. Anal.}, {\bf 170} (2000), 307--355. 

\bibitem{Flu}
M.~Flucher, {\em Extremal functions for the Trudinger-Moser inequality in 2 dimensions}, Comm. Math. Helv. {\bf 67} (1992), 471--479.

\bibitem{Ge1}
P.~G\'erard, {\em Oscillations and concentration effects in semilinear dispersive wave equations}, J. Funct. Anal. {\bf 133} (1996), 50--68.

\bibitem{Ge2}
P.~G\'erard, {\em Description du d\'efaut de compacit\'e de l'injection de Sobolev}, ESAIM Control Optim. Calc. Var. 3 (1998), 213--233 (electronic, URL: http://www.emath.fr/cocv/).


\bibitem{GMSTY}
M.~Garc\'ia-Huidobro, R.~Man\'asevich, J.~Serrin, M.~Tang and C.~S.~Yarur, 
{\em Ground states and free boundary value problems for the n-Laplacian in n dimensional space.} J. Funct. Anal. {\bf 172} (2000), no. 1, 177--201. 

\bibitem{Hans} K. Hansson, {\em Imbedding theorems of {S}obolev type in potential theory}, {Math. Scand.}, {\bf 45} (1979), 77--102.

\bibitem{HMT} J. A. Hempel,  G. R. Morris and N. S. Trudinger, {\em On the sharpness of a limiting case of the {S}obolev imbedding theorem }, {Bull. Austral. Math. Soc.}, {\bf 3} (1970),  369--373. 

\bibitem{DlogSob}
S.~Ibrahim, M.~Majdoub and N.~Masmoudi, {\em Double logarithmic inequality with a sharp constant}, Proc. Amer. Math. Soc. {\bf 135} (2007), no. 1, 87--97.

\bibitem{ScatBlow} S.~Ibrahim, N.~Masmoudi and K.~Nakanishi, {\em Scattering threshold for the focusing nonlinear Klein-Gordon equation}, preprint, arXiv:1001.1474, to appear in Analysis\&PDE. 

\bibitem{jaffard} S. Jaffard, {\em Analysis of the lack of compactness in the critical Sobolev embeddings},  J. Funct. Anal., {\bf 161} (1999),  384--396.

\bibitem{Yodovich61}
V.~I. Judovi{\v{c}}.
\newblock Some estimates connected with integral operators and with solutions
  of elliptic equations.
\newblock {\em Dokl. Akad. Nauk SSSR}, 138:805--808, 1961.


\bibitem{Lions-II} P.-L. Lions, {\em The concentration-compactness principle in the calculus of variations. {T}he limit case. {II}}, Rev. Mat. Iberoamericana, {\bf 1} (1985), 45--121. 

\bibitem{Lions1}
P.L.~Lions, {\em The concentration-compactness principle in the calculus of variations. The limit case. I.},  Rev. Mat. Iberoamericana 1 (1985), no. 1, 145--201.

\bibitem{Lions2}
P.L.~Lions, {\em The concentration-compactness principle in the calculus of variations. The locally compact case, I}, Ann. Inst. H. Poincare Anal. Non Lin\'eaire 1 (1984), 109--145.

\bibitem{Lions-IHP2} P.-L. Lions, {\em The concentration-compactness principle in the calculus of variations. The locally compact case. II}, Ann. Inst. H. Poincar\'e Anal. Non Lin\'eaire, {\bf 1} (1984), 223--283.


\bibitem{MP-PAMS} J. Mal{\'y} and L. Pick, {\em An elementary proof of sharp Sobolev embeddings}, Proc. Amer. Math. Soc., {\bf 130} (2002), 555--563.

\bibitem{Moser71}
J.~Moser, {\em A sharp form of an inequality of N. Trudinger}, Ind.
Univ. Math. J. {\bf 20}(1971), pp. 1077-1092.

\bibitem{Po}
S.~I.~Pohozhaev, {\em The Sobolev embedding in the case} $pl = n$,
Proc. Tech. Sci. Conf. on Adv. Sci. Research 1964-1965, Mathematics
Section, 158-170, Moskov. `Energet. Inst. Moscow, 1965.

\bibitem{Orlicz-Book} M.-M. Rao and Z.-D. Ren, {\em Applications of {O}rlicz spaces}, Monographs and Textbooks in Pure and Applied Mathematics, {\bf 250},
Marcel Dekker Inc., 2002.

\bibitem{Ruf}
B.~Ruf, {\em A sharp Trudinger-Moser type inequality for unbounded
domains in }$\mathbb R\sp 2$, J.~Funct.~Anal. {\bf 219} (2005),
no.~2, 340--367.

\bibitem{RS} T. Runst and W. Sickel, {\em Sobolev spaces of fractional order, {N}emytskij operators, and nonlinear partial differential equations}, {de Gruyter Series in Nonlinear Analysis and Applications}, Vol.{3}, Berlin, 1996.

\bibitem{Strauss} W.-A. Strauss, {\em Existence of solitary waves in higher dimensions}, Comm. Math. Phys., {\bf 55} (1977), 149--162.

\bibitem{Struwe88} M. Struwe, {\em Critical points of embeddings of {$H^{1,n}_0$} into {O}rlicz spaces}, {Ann. Inst. H. Poincar\'e Anal. Non Lin\'eaire}, {\bf 5} (1988), 425--464.

\bibitem{Trudinger67}
N.S.~Trudinger, {\em On imbedding into Orlicz spaces and some
applications}, J. Math. Mech. {\bf 17}(1967), pp. 473-484.

\end{thebibliography}
 

\end{document}